\documentclass{amsart}
\usepackage{amsmath}
\usepackage{bbm}
\usepackage{paralist}
\usepackage{graphics} 
\usepackage{graphicx} 
\usepackage{color}
\usepackage[pagewise]{lineno}
\usepackage[colorlinks=true]{hyperref}
\hypersetup{urlcolor=blue, citecolor=red}
\newtheorem{theorem}{Theorem}[section]

\newtheorem{lemma}[theorem]{Lemma}
\newtheorem{proposition}{Proposition}

\newtheorem{remark}{Remark}

\newcommand{\vv}{v}      %growth rate / vector field
\newcommand{\V}{V}		% Lyapunov function
\newcommand{\state}{\mathcal{X}}		% State Space 
\newcommand{\uu}{\mathcal{U}}		 % Markov semigroup
\newcommand{\gen}{\mathcal{L}}  % Inf. generator
\newcommand{\pp}{{\mathbb{P}}}	% Probability
\newcommand{\ee}{{\mathbb{E}}}	% Expectation
\newcommand{\1}{\mathbbm{1}}	% Indicator
\newcommand{\jrt}{\Lambda}   % jump rate
\newcommand{\bbR}{{\mathbb{R}}}
\newcommand{\bbN}{{\mathbb{N}}}
\newcommand{\rmd}{{\mathrm{d}}}
\begin{document}
%%%%%%%%%%%%%%%%%%%%%%%%%%%%%%%%%%%%%%%%%%
\title[Semistochastic Convergence Rates] 
      {Convergence Rates for Semistochastic Processes}
      %%%%%%%%%%%%%%%%%%%%%%%%%%%%%%%%%%%%%%%%%%
\author[James Broda and Alexander Grigo and Nikola P.\ Petrov]{}
\subjclass{Primary: 34F05, 60J25, 92D25; Secondary: 60Gxx, 92Bxx.}
%%%%%%%%%%%%%%%%%%%%%%%%%%%%%%%%%%%%%%%%%%
 \keywords{Semistochastic process, minorization, convergence rate, dynamics of carbon, 
 random catastrophe, ecosystem disturbance} 
%%%%%%%%%%%%%%%%%%%%%%%%%%%%%%%%%%%%%%%%%%
%%%%%%%%%%%%%%%%%%%%%%%%%%%%%%%%%%%%%%%%%%
%%%%%%%%%%%%%%%%%%%%%%%%%%%%%%%%%%%%%%%%%%%
\begin{abstract}
We study processes that consist of deterministic evolution punctuated at random times by disturbances with random severity; we call such processes semistochastic.  Under appropriate assumptions such a process admits a unique stationary distribution.  We develop a technique for establishing bounds on the rate at which the distribution of the random process approaches the stationary distribution. An important example of such a process is the dynamics of the carbon content of a forest whose deterministic growth is interrupted by natural disasters (fires, droughts, insect outbreaks, etc.).
\end{abstract}
%%%%%%%%%%%%%%%%%%%%%%%%%%%%%%%%%%%%%%%%%
%%%%%%%%%%%%%%%%%%%%%%%%%%%%%%%%%%%%%%%%%%
 \email{jbroda@wlu.edu}
 \email{alexander.grigo@ou.edu}
 \email{npetrov@ou.edu}
%%%%%%%%%%%%%%%%%%%%%%%%%%%%%%%%%%%%%%%%%%
%%%%%%%%%%%%%%%%%%%%%%%%%%%%%%%%%%%%%%%%%%
\thanks{J.B.\ and N.P.P.\ were partially supported by NSF grant DMS-0807658.  
A.G.\ was partially supported NSF grant DMS-1413428.  
N.P.P.\ was also generously supported by the Nancy Scofield Hester Presidential Professorship.  
We thank Martin Oberlack for useful suggestions.}
%%%%%%%%%%%%%%%%%%%%%%%%%%%%%%%%%%%%%%%%%%
%%%%%%%%%%%%%%%%%%%%%%%%%%%%%%%%%%%%%%%%%%
\thanks{$^*$ Corresponding author.  Present address: 
Department of Mathematics, Washington and Lee University, Lexington, VA 24450, USA}
%%%%%%%%%%%%%%%%%%%%%%%%%%%%%%%%%%%%%%%%%%
\noindent
{\em This article has been published in a revised form in {\em Discrete and Continuous Dynamical Systems - Series B} [\url{https://doi.org/10.3934/dcdsb.2019001}]. This version is free to download for private research and study only. Not for redistribution, re-sale or use in derivative works.}
\par \medskip
%%%%%%%%%%%%%%%%%%%%%%%%%%%%%%%%%%%%%%%%%%
\maketitle
%%%%%%%%%%%%%%%%%%%%%%%%%%%%%%%%%%%%%%%%%%
\centerline{\scshape James Broda$^*$, Alexander Grigo and Nikola P.\ Petrov}
\medskip
{\footnotesize
%%%%%%%%%%%%%%%%%%%%%%%%%%%%%%%%%%%%%%%%%%
\centerline{ Department of Mathematics}
   \centerline{University of Oklahoma}
   \centerline{Norman, OK, 73019, USA} 
} %%%%%%%%%%%%%%%%%%%%%%%%%%%%%%%%%%%%%%%%%%
\bigskip
%%%%%%%%%%%%%%%%%%%%%%%%%%%%%%%%%%%%%%%%%
\section{Introduction} 
%%%%%%%%%%%%%%%%%%%%%%%%%%%%%%%%%%%%%%%%%%
This line of research began due to a question from an ecologist: 
How should one model the carbon content of an ecosystem that experiences randomly occurring catastrophes of random severity? 
The role of disturbances such as droughts, forest fires, and insect outbreaks on the dynamics of carbon has been discussed in \cite{carbon1}, \cite {carbon2}, \cite{carbon3}, and~\cite{carbon4}.  In the absence of disturbances, the amount of carbon in an ecosystem increases naturally due to photosynthesis and eventually approaches the carrying capacity of the ecosystem.  On occasion, however, an extreme event results in significant destruction of an ecosystem and consequently a drastic reduction in the amount of carbon stored in the ecosystem.  
\par
In order to model the  carbon content of an ecosystem, continuous time continuous state space 
semistochastic processes were studied by Leite, Petrov, and Weng in \cite{np1} and formulae were derived for the densities of the corresponding stationary distributions.  
Semistochastic processes like the one studied in~\cite{np1} are a particular case of the so-called piecewise deterministic Markov processes (PDMPs).  
The dynamics of PDMPs lacks a diffusive component, and has been used in applications 
to growth-fragmentation processes, storage models, exposure to contaminants, 
communication networks, among others 
(for recent results and references see, e.g., \cite{Malrieu2015,BertailCT2010,LaurencotP2009,CalvezDG2012}).  
A general framework for studying PDMPs has been developed by Davis \cite{Davis1984} 
(see also his book \cite{Davis1993}).  
\par
Given that a stationary distribution exists (and can be calculated explicitly as in~\cite{np1}), 
a natural question is: At what rate does the process approach its stationary distribution?
If the time-dependent distributions are absolutely continuous, then one approach to resolving this would be to study the evolution of the corresponding time-dependent densities 
which is governed by an integro-differential PDE.  
An alternative and more general approach is to analyze the time-evolution of the corresponding distributions through their action on observables, which is the picture dual to using the integro-differential PDE.  
In this  paper we adopt the latter approach which works also for distributions that are not absolutely continuous.
\par
In this paper, we utilize purely probabilistic methods to establish explicitly computable bounds on convergence rates; consequently, our methods for determining convergence rates, are quite different from the methods used in \cite{np1} to develop exact formula for the stationary distributions.  The methods we use originated in the study of discrete-time Markov chains and are based on establishing a combination of minorization and drift conditions.  These approaches go back to Doeblin, and appear in various forms in \cite{tweed2,num2,ros1,ros2,RobertsT2000}.  
For an overview of more recent work see, e.g., \cite{TeelSS2014,Malrieu2015}.  
Roughly speaking, minorization conditions are bounds on the probability of transitioning 
in one step from any initial value to some specified region in the state space.  
Drift conditions, on the other hand, need to be applied when the state space is unbounded 
and the stochastic process may drift arbitrarily far away.  
The drift conditions allow us to control the process in some bounded set 
while also keeping track of the probability of the the process drifting out of the set.  
For detailed description of the minorization and drift conditions see Section~\ref{sec:minorization}, 
in particular, equations \eqref{eq:minor-1} and~\eqref{eq:drift-condition}.  
\par 
While the problem of modeling the carbon content of an ecosystem was the original inspiration for this project, our work can be applied to any problem admitting a semistochastic model, i.e., population dynamics, optimal harvesting, 
virus reproduction, and some of the problems mentioned previously in this Introduction.  
\par
What follows is a brief introduction to the concept of a semistochastic process.
By semistochastic process we mean a continuous-time, continuous-state process $\{X_t\}$, with state space $\state$,  which consists of intervals of deterministic evolution punctuated by random events.  The random events we typically consider occur on time-scales much larger than the typical inter-event time, and are modeled as instantanteous events.  These processes are assumed to be doubly-stochastic in the sense that there is a random severity associated to each event as well as the random time at which it occurs.  Consequently, these types of processes are quite different from other types of stochastic processes and can be used to model dynamical systems that lack conservation laws, see \cite{Davis1984, rom}.  Semistochastic processes do share some common features with what are typically referred to as stochastic clearing processes, see \cite{whitt}.  A clearing process, however, consists of epochs of random growth punctuated by instantaneous returns to the initial value once a critical threshhold is reached.  A semistochastic process replaces the random growth in a clearing process with determnistic growth and replaces the deterministic ``clearing" with randomly occurring disturbances.
\par 
The operator-theoretic framework which we set up to study the dynamics of semistochastic processes applies equally well to both scalar- and vector-valued stochastic processes, but we restrict our attention to scalar processes when deriving bounds on convergence rates.  In the scalar case we are thus interested in sample paths that are piecewise continuous, right-continuous, and have left-hand limits almost surely (\textit{c\`{a}dl\`{a}g}).  
We furthermore focus our attention on disturbances that correspond to a diminishing in value.  
The model that one should have in mind is the carbon content in a forest that grows deterministically 
and is interrupted and random times by natural disasters which reduce the amount of carbon.  
We should note that the techniques we use can be adapted to handle more general disturbances as well.  
\par
In the time between two consecutive disturbances, 
$\{X_t\}$ evolves deterministically, 
governed by the autonomous ordinary differential equation 
\begin{equation}  
\frac{\rmd}{\rmd t} x(t) = \vv(x(t)) \ .
\end{equation}
To describe when the disturbances occur, we specify a rate function $\jrt(x)$ which is a measure of the instantaneous rate of occurrence of the disturbances.  We refer to $\jrt(x)$ as the \textit{jump rate} for the process.  
%%%%%%%%%%%%%%%%%%%%%%%%%%%%%%%%%%%%%%%%%%
%%%%%%%%%%%%%%%%%%%%%%%%%%%%%%%%%%%%%%%%%%
\par
Our problem gains another element of randomness from the varying severity of the disturbances.  In order to describe this severity, we introduce random variables $Y_n^-$ and $Y_{n}$ corresponding to the $n^{th}$ pre- and post- disturbance values, respectively.
If the $n^{th}$ disturbance occurs at time $T$, then $Y_n^-$ and $Y_{n}$ are defined via
\[
Y_n^- := \lim_{t \nearrow T} X_t \, , \qquad Y_n := \lim_{t \searrow T} X_t \ .
\] 
In the simplest case, we can then model the severity by stipulating a multiplicative relation between $Y_n^-$ and $Y_{n}$.  An additional random variable, $A_n$ is then defined by setting
\begin{equation*}
Y_n=A_nY_{n}^- \,.
\end{equation*} 
\begin{figure}[ht] 
\centerline{
\includegraphics[scale=.65]{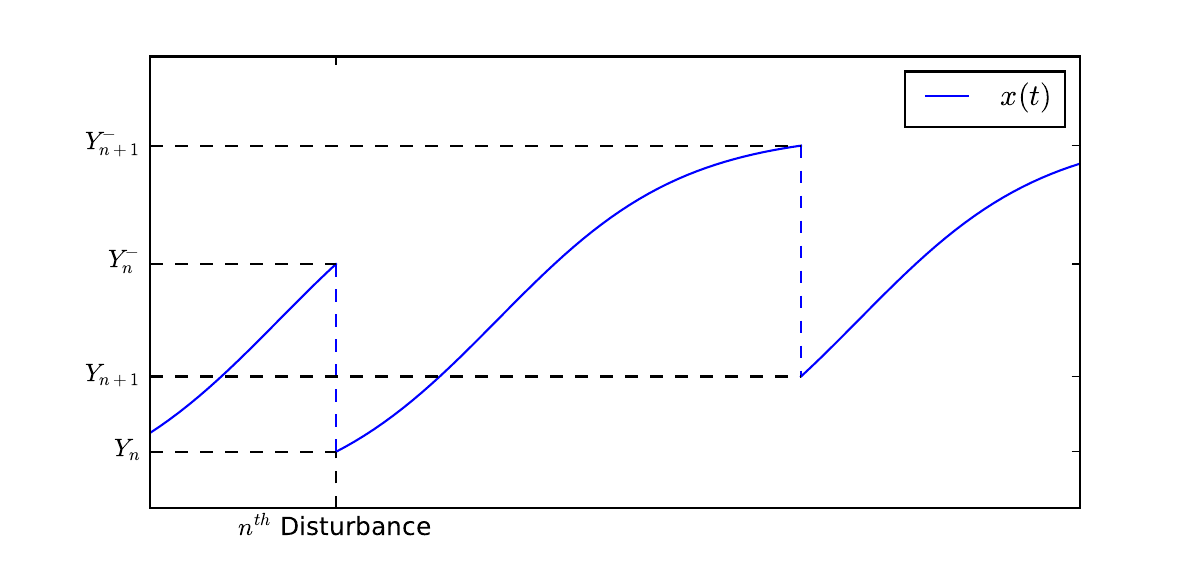}
}
\caption{Schematic for pre- and post- disturbance levels.
}
\label{fig:prepost}
\end{figure}
\par
Having specifed the types of processes we propose to study, we now mention some works that study similar processes, but usually under different assumptions or with different goals.  The most common difference is due to the fact that most of the research on semistochastic processes is concerned with population dynamics, and demographers generally study processes with discrete state-spaces.
\par
An interesting application of semistochastic processes  is proposed by Bartoszy\'{n}ski in \cite{rabies} to model the development of the rabies virus in an infected host.  In this model, the population of the virus naturally decreases exponentially due to the immunological response of the host, but also has random upward jumps due to the viral life cycle.  The state space of the model  Bartoszy\'{n}ski constructs is discrete and the occurrence of jumps is allowed to depend on the current population.
\par
Continuous-time and continuous state space processes subject to random catastrophes are studied by Gripenberg in \cite{grip1}. Gripenberg derives an expression for stationary distributions using a limit theorem from \cite{ath1} based on the concept of Harris recurrence.  There is a connection between the type of recurrence condition that is established in \cite{grip1} and the minorization conditions that we establish, however the issue of convergence rates is not addressed by Gripenberg.  Biedrzycka and Tyran-Kam\'{i}nska \cite{BiedrzyckaTK2016} use operator-theoretic techniques to address the question of existence of invariant densities for similar processes.  
\par
Hanson and  Ryan in \cite{han1} and \cite{han2} examine optimal harvesting problems of populations governed by similar processes with discrete state spaces, though they do not allow for the randomization of the severity of disturbances.  They do, however, consider the possibility of populations experiencing both sudden decreases (jumps down) and sudden increases (jumps up).  With slight modifications, the results we present can also be applied in these situations.  Hanson and Tuckwell also study similar processes with discrete state spaces in \cite{hant1, hant2, hant3}, though their focus is generally on the computation of extinction times.  The problem of determining extinction times in semistochastic models is addressed more recently by Cairns in~\cite{pop3}.
\par
Transient distributions in discrete time discrete state space processes, 
e.g., a birth/immigration-death process with binomial catastrophes, 
have been studied in \cite{EconomouF2008,KapodistriaPR2016}.  
\par
Recent work \cite{BourgeronDE2014,AzaisG2016,AzaisMG2016} develops 
inverse techniques for growth-fragmentation phenomena (like cell division and polymerization) 
that are quite similar to our model.  In particular, these authors propose a method 
for calibrating the jump rate from empirical measurements.  
\par
The plan of our paper is the following: in Section \ref{sec:main} we state our main results on convergence rates, in Section \ref{sec:proofs} we provide proofs of our results by establishing a combination of minorization and drift conditions, we conclude with Section \ref{sec:examples} in which we apply our results to concrete examples.
%%%%%%%%%%%%%%%%%%%%%%%%%%%%%%%%%%%%%%%%%
%%%%%%%%%%%%%%%%%%%%%%%%%%%%%%%%%%%%%%%%%
\section{Statements of the main results\label{sec:main}} 
We start by revisiting the properties of the semistochastic process $\{X_t\}$.
In the time between two consecutive disturbances, 
$X_t$ evolves deterministically, 
governed by the autonomous ordinary differential equation 
\begin{equation}\label{eq:growth}
\frac{\rmd}{\rmd t} x(t)  =  \vv(x(t)) \ .
\end{equation}
We assume throughout that the vector field $v(x)$ admits a unique global solution to~\eqref{eq:growth} 
for any initial value $x(0)$; global Lischitz continuity of $v(x)$ is a sufficient condition.  
%%%%%%%%%%%%%%%%%%%%%%%%%%%%%%%%%%%%%%%%%%
The corresponding flow of \eqref{eq:growth} with initial condition $x(0)=x_0$ is denoted by $\phi^t(x_0)$, and the time duration needed to deterministically evolve from  $x_0$ to $x_1>x_0$ is denoted by 
$\psi(x_0,x_1)$.  Thus, in the absence of disturbances, we have
\begin{equation}  \label{eq:phi-psi}
x_1 = \phi^t(x_0) \quad \Longleftrightarrow \quad t = \psi(x_0,x_1) \ .
\end{equation}
We assume that the occurrences of disturbances have a distribution 
related to a jump rate parameter $\jrt(x)$ given by 
%%%%%%%%%%%%%%%%%%%%%%%%%%%%%%%%%%%%%%%%%%
\[
\pp(\text{disturbance occurs in }(t, t+\Delta t] \, | \, X_t = x)= \jrt(x) \, \Delta t + o (\Delta t) 
\]
%%%%%%%%%%%%%%%%%%%%%%%%%%%%%%%%%%%%%%%%%%
as $\Delta t \to 0$.  
Furthermore, to determine the severity of individual disturbances, we define the multiplier density, 
$\rho(x,\alpha)$, with the property that for any $a \in (0, 1)$
\begin{equation}\label{eq:multiplier-density}
\pp\left(Y_n \leq ax \, | \, Y_n^- =x\right)=\int_0^a \rho(x,\alpha) \, \rmd \alpha \ ,
\end{equation}
where $Y_n$ is the $n^{th}$ post-disturbance random variable and $Y_n^-$ is the $n^{th}$ pre-disturbance random variable.  It will be convenient  to introduce the quantity $\zeta(x)$ to denote the expected fractional loss resulting from a single disturbance,
\begin{equation}\label{eq:zeta}
\zeta(x) 
:=    
\int_0^1 \rho(x,\alpha)\left(1  -  \alpha \right)\, \rmd \alpha   \in (0, 1)  \ .
\end{equation}
Thus larger values of $\zeta(x)$ correspond to an expectation of more severe disturbances and the limiting value $\zeta=0$ would result in purely deterministic growth.
\par 
All of this can be consolidated by specifying the infinitesimal generator 
$\gen$ of~$\{X_t\}$.  
The action of $\gen$ on observables $f$ from the appropriate Banach space is then given by 
\begin{equation}\label{eq:gen}
[\gen f](x)
=
f'(x) \vv(x)  +  \jrt(x)\int_0 ^1  \rho(x,\alpha) \,  [f(\alpha x) \, - \, f(x) ] \, \rmd \alpha \ .
\end{equation}
Corresponding to the generator $\gen$ is a Markov semigroup $\uu^t$ 
which can be specified by its action on observables: 
\[
[\uu^{t}f](x) = \ee[f(X_{t})|X_0=x] \ .
\]
If the distribution of $X_0$ is $\mu_0$, 
then the distribution of $X_t$ is 
\[
\mu_t := \mu_0 \uu^t \ .
\]

In order to quantify the convergence rates, we use the total variation distance~$d_{\mathrm{TV}}$ 
defined for any two distributions, $\mu_1$ and $\mu_2$, by 
\[
d_{\mathrm{TV}}(\mu_1, \mu_2) := \sup_{\, 0 \leq f(x) \leq  1} |\mu_1(f)-\mu_2(f)|.
\]

We are now ready to state our first result.
%%%%%%%%%%%%%%%%%%%%%%%%%%%%%%%%%%%%%%%%%%%%%%%%%%%%%%%%%%%%%%%%%%%%%%%%%%%%%%%%%
\begin{theorem}\label{thm:bounded-state-rate}
Let $\{X_t\}$ be a semistochastic process with generator \eqref{eq:gen} 
on the state space $\state = [0,k]$, satisfying 
%and furthermore satisfying:
\begin{itemize}
\item[(i)] 
$0 <  \lambda_*  \leq  \jrt(x)  \leq  \lambda^{*}  <  \infty$ for all $ x  \in \state $, 
for some constants $ \lambda_*$ and~$\lambda^{*}$, 
\item[(ii)] 
$\rho(x,\alpha)  \geq  \rho_{*} $ for all $x \in \state$ and $ \alpha  \in  [0, 1]$, 
for some constant $\rho_* > 0$, 
\item[(iii)]
the function $\vv$ is non-negative,  
Lipschitz, and $\vv(0) \neq 0$ with $\vv(x)=0$ for at most finitely many $x$.
\end{itemize}
Then $\{X_t\}$ converges exponentially fast to its unique stationary distribution $\pi$. 
Namely, for any time increment $\Delta t > 0$, and any initial distribution $\mu_0$, 
%%%%%%%%%%%%%%%%%%%%%%%%%%%%%%%%%%%%%%%%%%
\begin{equation}  \label{eq:dTV-bounded}
d_{\mathrm{TV}}\left( \mu_t , \pi \right)  
\leq  
\left(1-\epsilon_{\Delta t}\right)^{\lfloor t /\Delta t \rfloor} \ , 
\end{equation}
where 
\begin{equation}  \label{eq:epsilon-delta-t}
\epsilon_{\Delta t} 
:= 
\frac{  \rho_{*}  \Phi  \lambda_*   \exp (-\lambda^{*}  \Delta t) } { k }   \ ,
\end{equation}
\begin{equation}\label{eq:Phi}
\Phi 
:= \int_ 0^{\phi^ {\Delta t} (0)} [\Delta t - \psi(0, z)] \, \rmd z  \ ,
\end{equation}
and $\phi$ and $\psi$ defined in~\eqref{eq:phi-psi}. 
%%%%%%%%%%%%%%%%%%%%%%%%%%%%%%%%%%%%%%%%%%
\end{theorem}
%%%%%%%%%%%%%%%%%%%%%%%%%%%%%%%%%%%%%%%%%%%%%%%%%%%%%%%%%%%%%%%%%%%%%%%%%%%%
\begin{remark}  \label{rem:choices-bounded}
The bound on the rate of convergence given by \eqref{eq:dTV-bounded} 
can be optimized by choosing a value of $\Delta t$ that makes this bound tighter.  
It is intuitively reasonable to expect that a value of $\Delta t$ that minimizes 
$\left(1-\epsilon_{\Delta t}\right)^{\lfloor t /\Delta t \rfloor}$ should exist.  
If $\Delta t$ is too small, then a disturbance is unlikely to occur 
in the short time interval of length~$\Delta t$.  
On the other hand, if $\Delta t$ is chosen too large, then we do not control the process 
over a long time interval during which many disturbances of varying severity may occur 
which would make it impossible for us to use any features of the deterministic growth.  
Put differently, the optimal value of $\Delta t$ should correspond to appropriate balance 
between the stochastic and the deterministic components of the dynamics 
-- for $\Delta t$ too small, we observe only the deterministic component, 
while for $\Delta t$ too large, we observe mainly the stochastic one.  
The mathematical intuition behind the existence of an ``optimal'' value of $\Delta t$ 
can be seen from the text in Section~\ref{sec:bounds-bounded} and, in particular, 
from Figure~\ref{fig:minor-domain}.  
\end{remark}
%%%%%%%%%%%%%%%%%%%%%%%%%%%%%%%%%%%%%%%%%%%%%%%%%%%%%%%%%%%%%%%%%%%%%%%%%%%%
The general strategy of the proof of Theorem~\ref{thm:bounded-state-rate} is the following 
(the complete proof is given in Section~\ref{sec:proofs}).  
We begin by discretizing the process by fixing a $\Delta t >0$ 
and studying the resulting discrete-time Markov chain 
with transition kernel~$\uu^{\Delta t}$. 
We then construct a uniform minorization for this discretization, 
which yields well-known exponential bounds on the convergence rates.  
It remains only to apply the well-known fact that 
$d_{\mathrm{TV}}\left( \mu_t , \pi \right)$ is monotonically decreasing in~$t$ 
to obtain bounds for the original continuous time-process.  
\par
While the restriction $\vv(0)>0$ may seem unusual for biological models, it is a reasonable assumption for the carbon content problem since even in the event of a complete catastrophe, there is regrowth.  The specific case of $\vv(x) = 1-x$ with state space $\state = [0, 1]$ was considered in \cite{np1} as a model for carbon content in an ecosystem and meets all conditions of our Theorem \ref{thm:bounded-state-rate}.
In establishing a uniform minorization, it is essential that the state space be bounded.  While this is the case for most applications, such as the carbon content problem, it is mathematically restrictive. Though the proof requires additional work, we are also able to state a result for unbounded state semistochastic processes.  
%%%%%%%%%%%%%%%%%%%%%%%%%%%%%%%%%%%%%%%%%%%%%%%%%%%%%%%%%%%%%%%%%%%%%%%%%%%%%%%%%%
\begin{theorem}\label{thm:unbounded-state-rate}
%%%%%%%%%%%%%%%%%%%%%%%%%%%%%%%%%%%%%%%%%%
Let $\{X_t\}$ be a semistochastic process with generator \eqref{eq:gen} 
on the state space $\state = [0,\infty)$, satisfying 
\begin{itemize}
\item[(i)] 
$0  <   \lambda_* \leq  \jrt(x)  \leq  \lambda^{*}  <  \infty$ for all $ x  \in  \state $, 
for some constants $ \lambda_*$ and~$\lambda^{*}$, 
\item[(ii)] 
$\rho(x,\alpha)  \geq  \rho_{*} $ for all $x \in \state$ and $ \alpha  \in  [0, 1]$, 
for some constant $\rho_* > 0$, 
\item[(iii)]
$\zeta(x) \geq \zeta_*$ for all $x \in \state$, for some constant $\zeta_* > 0$, 
\item[(iv)]
the function $\vv$ is Lipschitz, satisfies 
\begin{equation}  \label{eq:vv}
0 \leq v(x) \leq v^* = \mathrm{const} \ , \qquad v(0) \neq 0 \ , 
\end{equation}
and vanishes for at most finitely many~$x$.  
\end{itemize}
Then $\{X_t\}$ has a unique stationary distribution $\pi$ to which it converges at an exponential rate.  Namely, for any initial distribution $\mu_0$ and any $\Delta t>0$, the estimate 
\begin{equation}  \label{eq:dTV-unbounded}
d_{\mathrm{TV}}\left( \mu_t  ,  \pi \right) 
\leq 
\left(2+ \frac{b}{1-\beta} + \ee_{\mu_0}[X_0] \right) 
\, (1-\epsilon_{\Delta t,\kappa})^{r \lfloor t / \Delta t \rfloor}\ 
\end{equation}
holds with $\Phi$ given by~\eqref{eq:Phi}, 
\[
\epsilon_{\Delta t, \kappa} :=    \frac{ \rho_{*}  \Phi  \zeta_* \lambda_* \exp (-\lambda^{*}  \Delta t)} {\kappa}   \ ,
\]
\[
\beta := e^{-\lambda_* \zeta_* \Delta t} \ , 
\qquad 
b := \frac{\vv^{*}}{\lambda_* \zeta_*} 
\left( 1 \ - \ e^{-\lambda_* \zeta_* \Delta t} \right) \ , 
\]
\begin{equation}\label{eq:thetas}
\theta := \frac{1  +  2b  +  \kappa  \beta}{1  +  \kappa} \ , 
\qquad 
\Theta := 1 + 2(\beta  \kappa +  b) \ , 
\end{equation}
\begin{equation}  \label{eq:r}
r := \frac{\ln\theta}{\ln\frac{\theta(1-\epsilon_{\Delta t,\kappa})}{\Theta}}  
= \frac{\ln\frac{1}{\theta}}{\ln\frac{1}{\theta}+\ln\Theta+\ln\frac{1}{1-\epsilon_{\Delta t,\kappa}} }
\in (0,1) \ , 
\end{equation}
where $\kappa$ can be chosen to be any number satisfying 
\begin{equation}  \label{eq:cond-kappa}
\kappa > \frac{2b}{1-\beta} \ .
\end{equation}
\end{theorem}
%%%%%%%%%%%%%%%%%%%%%%%%%%%%%%%%%%%%%%%%%%
\begin{remark}
From the argument in Section~\ref{sec:bounds-unbounded}, 
it can be easily seen that 
the assumptions on the lower bounds on $\Lambda(x)$ and $\zeta(x)$ 
in the statement of Theorem~\ref{thm:unbounded-state-rate} 
could be relaxed to 
$\inf_{x \in \state} \left[ \Lambda(x) \zeta(x) \right] > 0$.  
In this case the statement of Theorem~\ref{thm:unbounded-state-rate} 
remains unchanged if the product $\lambda_* \zeta_*$ 
is replaced by 
$\inf_{x \in \state} \left[ \Lambda(x) \zeta(x) \right]$.  
\end{remark}
%%%%%%%%%%%%%%%%%%%%%%%%%%%%%%%%%%%%%%%%%%
\begin{remark}  \label{rem:choices}
Note that the bound \eqref{eq:dTV-unbounded} on the rate of convergence 
depends on the choice of $\Delta t$ and $\kappa$ (cf.~Remark~\ref{rem:choices-bounded}).  
To obtain tight bounds, one can choose values of $\Delta t$ and $\kappa$ that minimize 
$( 1 - \epsilon_{\Delta t,\kappa} )^{r / \Delta t}$, which can be done numerically 
as shown in the examples in Section~\ref{sec:examples}.  
Intuitively, the optimal value of $\Delta t$ corresponds to balancing 
the deterministic and the stochastic components of the process.  
The optimal value of $\kappa$, on the other hand, balances the rate of convergence 
while staying in $[0,\kappa]$ with the time it takes to re-enter the region $[0,\kappa]$ after leaving it, 
as one can see from the proof in Section~\ref{sec:bounds-unbounded}.  
This can be clearly seen in the numerical example in Section~\ref{sec:example-unbounded} 
and, in particular, in Figures \ref{fig:rate5} and~\ref{fig:rate6}.  
\end{remark}
%%%%%%%%%%%%%%%%%%%%%%%%%%%%%%%%%%%%%%%%%%
The details of the proof of Theorem~\ref{thm:unbounded-state-rate} are provided in Section~\ref{sec:proofs}.
The strategy of the proof is similar to that of Theorem \ref{thm:bounded-state-rate}.  The primary difference is that in this case we cannot establish a uniform minorization.  Instead, we establish a combination of drift and minorization conditions which enables us to apply a result of Rosenthal \cite{ros1} (see also \cite{ros2,RobertsT2000}) 
to produce the desired bounds on convergence rates. 
%%%%%%%%%%%%%%%%%%%%%%%%%%%%%%%%%%%%%%%%%%%%%%%%%%%%%%%%%%%%%%%%%%%%%%%%%%%%%%%%%%
%%%%%%%%%%%%%%%%%%%%%%%%%%%%%%%%%%%%%%%%%%%%%%%%%%%%%%%%%%%%%%%%%%%%%%%%%%%%%%%%%%

\section{Proofs of the main results\label{sec:proofs}}
%%%%%%%%%%%%%%%%%%%%%%%%%%%%%%%%%%%%%%%%%%
The proofs of Theorems \ref{thm:bounded-state-rate} and 
\ref{thm:unbounded-state-rate} follow similar ideas, 
so we develop them in parallel.  
In Section~\ref{sec:inequalities} we derive some inequalities 
about the Markov semigroup~$\uu^t$ and relate the rates of convergence 
of the continuous-time semigroup $\uu^t$ and of its discretization 
(see \eqref{eq:Q-U} below) to the stationary distribution.  
We define the drift condition and state some results on minorization 
in Section~\ref{sec:minorization}.  
The bounds of the rates of convergence for bounded and unbounded 
state space are derived in Sections \ref{sec:bounds-bounded} 
and~\ref{sec:bounds-bounded}, respectively.  
%%%%%%%%%%%%%%%%%%%%%%%%%%%%%%%%%%%%%%%%%%
\subsection{Some useful inequalities\label{sec:inequalities}}
%%%%%%%%%%%%%%%%%%%%%%%%%%%%%%%%%%%%%%%%%%
We separate the infinitesimal generator into two components, 
$\gen=\gen_0+\gen_1$, where 
\[
[\gen_0f](x)= f'(x)v(x)-\jrt(x)f(x) 
\]
corresponds to deterministic evolution plus a loss term, and 
\[
[\gen_1f](x)=\jrt(x)\int_0^1 \rho(x,\alpha) f(y)  \, \rmd \alpha 
\] 
reflects the ``gain''.  
We introduce the {\em semistochastic survival function\/}, 
\begin{equation}\label{eq:survival}
S(t, x) := \exp\left(-\int_0^t  \jrt \left(\phi^s (x)\right) \, \rmd s \right) \ ,
\end{equation}
which represents the conditional probability of starting at $x$ 
and evolving deterministically for time $t$ 
with no occurrence of a disturbance.  
Then the sub-Markov semigroup $\uu_0$ generated by $\gen_0$ is 
\[
[\uu_0^t f](x)=S(t, x) \, f(\phi^t(x)) \ ,
\]
%%%%%%%%%%%%%%%%%%%%%%%%%%%%%%%%%%%%%%%%%%
which can be verified directly using that 
\[
\frac{\partial}{\partial t} S(t,x) 
= 
- \Lambda(\phi^t(x)) \, S(t,x) \ , \qquad 
\frac{\partial}{\partial t} f(\phi^t(x)) 
= 
v(\phi^t(x)) \, f'(\phi^t(x)) \ .
\]
%%%%%%%%%%%%%%%%%%%%%%%%%%%%%%%%%%%%%%%%%%
The Markov semigroup $\uu^t$ can be computed iteratively, as given in the following 
%%%%%%%%%%%%%%%%%%%%%%%%%%%%%%%%%%%%%%%%%%%%%%%%%%%%%%%%%%%%%%%%%%%%%%%%%%%%%%%%%%%%
\begin{proposition} \label{prop:u-expansion}
Let $\uu^t$ be a strongly continuous Markov semigroup with infinitesimal generator $\gen$ and assume that $\gen = \gen_0 +\gen_1$, with $\gen_0$ generating the sub-Markov semigroup $\uu_0^t$. Then the action of $\uu^t$ on an observable $f$ can be decomposed into
\begin{equation*}
[\uu^{t} f](x) =
[\uu_{0}^{t}f](x)
+\int_{0}^{t} \left[\uu_{0}^{t-s} \left( \gen-\gen_0 \right) \uu^{s}f\right](x) \, \rmd s \ .
\end{equation*}
\end{proposition}
%%%%%%%%%%%%%%%%%%%%%%%%%%%%%%%%%%%%%%%%%%%%%%%%%%%%%%%%%%%%%%%%%%%%%%%%%%%%%%%%%%%
\begin{proof}
Let $0\leq s \leq t$, and recall that $\uu^0$ and $\uu_0^0$ are both identity operators. 
Then 
\begin{align*}
\int_{0}^{t}\left[ \uu_{0}^{t-s} (\gen-\gen_0 ) \uu^{s}f\right](x) \, \rmd s
&=
\int_{0}^{t}\left[  \frac{\rmd}{\rmd s}  \left(\uu_{0}^{t-s} \uu^s \right) f \right](x) \, \rmd s \\
&= 
\ \left[   \left(\uu_{0}^{0} \, \uu^t - \uu_{0}^{t} \, \uu^0  \right) f \right](x) 
= 
\left[\uu^t f\right](x) - \left[\uu_{0}^{t} f\right](x) \ .
\end{align*}
Solving for $\uu^t$ above yields the result. 
\end{proof}
%%%%%%%%%%%%%%%%%%%%%%%%%%%%%%%%%%%%%%%%%%%%%%%%%%%%%%%%%%%%%%%%%%%%%%%%%%%%%%%%%
Combining this with the expression \eqref{eq:gen} for $\gen$, we have 
\begin{equation}  \label{eq:ttt}
\begin{split}
[\uu^t f](x)
&=
\left[\uu_{0}^{t}f\right](x)  +  \int_{0}^{t} 
\bigl[\uu_{0}^{t-s} ( \gen-\gen_0) \uu^{s}f\bigr](x)\, \rmd s \\
&=
S(t, x) f(\phi^t(x))\\
&\quad +
\ \int_0^t \rmd s \,  S(t-s, x) \, \jrt \left(\phi^{t-s}(x)\right)
 \int P \left(\phi^{t-s}(x), \rmd y \right) \, \left[\uu^s  f\right](y) \ .
\end{split}
\end{equation}
Noticing that in \eqref{eq:ttt}, $\uu^t_0$ is positive, we obtain 
%%%%%%%%%%%%%%%%%%%%%%%%%%%%%%%%%%%%%%%%%%%%%%%%%%%%%%%%%%%%%%%%%%%%%%%%%%%%%%%%%%%%%
\begin{lemma}\label{lemma:u-ineq}
If $\uu^t$ is a Markov semigroup with infinitesimal generator $\gen$ \eqref{eq:gen}, then 
\begin{align*}
\left[\uu^t f\right](x) 
\geq 
\int_0^t \rmd s \, S(t-s, x) \, \jrt \left(\phi^{t-s}(x)\right)
\int_0 ^1 \rmd \alpha \, \rho(x,\alpha) \, S(s, \alpha \phi^s( x)) \, f(\alpha \phi^s( x))  \ .
\end{align*}
\end{lemma}
%%%%%%%%%%%%%%%%%%%%%%%%%%%%%%%%%%%%%%%%%%
Next, we establish an inequality linking convergence rates for continuous-time Markov processes to their discretizations.  
We discretize the continuous-time process $\{X_t\}$ by sampling it at times that are separated by time increments 
of fixed specified size~$\Delta t$.  
The choice of a constant separation time $\Delta t$ allows for straightforward comparison between the continuous-time process $\{X_t\}$ and the discretized process $\{X_{n \, \Delta t}\}_{n \geq 0}$.  
The optimal value of $\Delta t$ (recall Remark~\ref{rem:choices-bounded}) 
can be selected in each particular example, as illustrated in Section~\ref{sec:examples}.  
%%%%%%%%%%%%%%%%%%%%%%%%%%%%
%%%%%%%%%%%%%%%%%%%%%%%%%%%%
\begin{lemma}\label{lemma:continuous-to-discrete}
Let $\pi$ denote the stationary distribution for a continuous-time Markov process $\{X_t\}$  with Markov semigroup \, $\uu^t$ and let $\Delta t > 0$ be a fixed time increment.  
If we set 
\begin{equation}  \label{eq:Q-U}
Q= \uu^{\Delta t} \ , 
\end{equation}
then for any initial distribution $\mu_0$ of~$X_0$, 
\begin{equation*}
d_{\mathrm{TV}}\left( \mu_0 \, \uu^t  ,  \pi \right)   \leq  d_{\mathrm{TV}}\left( \mu_0 \, Q^n  , \pi \right) \ ,
\end{equation*}
where $n=\lfloor t / \Delta t \rfloor$ is the greatest integer less than or equal to $t / \Delta t$\ .
\end{lemma}
\begin{proof}
Write $t  =  n\Delta t  + \tau$ for $0 \leq \tau < \Delta t$, 
then for any observable $f$ with $0 \leq f \leq  1$, 
\begin{align*}
\left|\mu_0 \, \uu^t f- \pi f\right| 
&= 
\left|\mu_0 \, \uu^{n\Delta t}\, \uu^{\tau} f  -   \pi f\right| 
= 
\left| \mu_0 \, \uu^{n\Delta t}\, \uu^{\tau}  f   -   \pi \uu^\tau f\right|\\
&\leq 
\sup_{|g|_{\infty}  \leq \ 1} \left|\mu_0 \, \uu^{n\Delta t}\, g   -  \pi g\right| 
= \ d_{\mathrm{TV}}( \mu_{0} \, Q^n , \pi ) \ ,
\end{align*}
where we used the invariance of $\pi$ and the fact that 
$ 0 \leq \uu^t f  \leq  1$.  
%%%%%%%%%%%%%%%%%%%%%%%%%%%%%%%%%%%%%%%%%%
\end{proof} 
%%%%%%%%%%%%%%%%%%%%%%%%%%%%%%%%%%%%%%%%%%%%%%%%%%%%%%%%
%%%%%%%%%%%%%%%%%%%%%%%%%%%%%%%%%%%%%%%%%%%%%%%%%%%%%%%%
%%%%%%%%%%%%%%%%%%%%%%%%%%%%%%%%%%%%%%%%%%%%%%%%%%%%%%%
\subsection{Minorization and drift condition\label{sec:minorization}}
%%%%%%%%%%%%%%%%%%%%%%%%%%%%%%%%%%%%%%%%%%
A Markov chain $X_n$ with transition kernel $Q$ on a state space $\state$ is said to satisfy a \emph{minorization condition\/} on a subset $A\subseteq \state$
if there is a probability measure $\eta$ on $\state$,  a positive integer $n_0$, and a number $\epsilon >0$ such that 
\begin{equation}  \label{eq:minor}
Q^{n_0}(x, B) \ \geq \  \epsilon \, \eta(B)
\end{equation}
for all $x\in A$ and for any measurable set $B$ of~$\state$.
By appropriately redefining $Q$, we can write this condition as 
\begin{equation}  \label{eq:minor-1}
\left[ Q f\right] (x) 
= \int Q(x, \rmd y) \, f(y) 
\geq 
\epsilon  \int f(y) \, \rmd \eta(y)  
\end{equation}
for any nonnegative observable $f$ and for all $x \in A$.  
If in these conditions the subset $A$ is the whole state space~$\state$, 
we say that $X_n$ admits a \emph{uniform minorization}.  
%%%%%%%%%%%%%%%%%%%%%%%%%%%%%%%%%%%%%%%%%%
The following theorem can be found in \cite{doob} or~\cite{tweed3}. 
%%%%%%%%%%%%%%%%%%%%%%%%%%%%%%%%%%%%%%%%%%
\begin{theorem}\label{thm:uniform-minorization}
If there exists an $n_0 \in \bbN$ such that the transition kernel $Q$ of a Markov chain on a state space $\state$ satisfies \eqref{eq:minor}
for all $x \in \state$ and any measurable set $B\subseteq \state$, 
then for any initial distribution $\mu_0$, 
the total variation distance to its unique stationary distribution $\pi$  satisifes
\[
d_{\mathrm{TV}}\left( \mu_0 \, Q^n , \pi \right) 
\leq 
\left(1-\epsilon \right)^{\lfloor n/n_0\rfloor} \ .
\]
\end{theorem}
%%%%%%%%%%%%%%%%%%%%%%%%%%%%%%%%%%%%%%%%%%
In the proof of Theorem~\ref{thm:unbounded-state-rate} 
we need to impose an additional condition.  
%%%%%%%%%%%%%%%%%%%%%%%%%%%%%%%%%%%%%%%%%%
A Markov chain $X_n$ with state space $\state$ satisfies a \emph{drift condition\/} 
if there exists a nonnegative function $\V:\state \mapsto \bbR_{\geq 0}$, 
a number $\beta < 1$, and some finite $b \in \bbR$ such that 
\begin{equation}  \label{eq:drift-condition}
\ee\left[ \V( X_1) | X_0 = x\right] \ \leq \ \beta \, \V(x) \ + \ b
\end{equation}
for all $x\in \state$.  The function $\V$ has sometimes been referred to as \textit{Lyapunov function\/} in the literature.  
%%%%%%%%%%%%%%%%%%%%%%%%%%%%%%%%%%%%%%%%%%
When a uniform minorization is unavailable, 
one can first establish a drift condition, 
and subsequently minorize on a subset $A$ of~$\state$, 
to obtain the following result proved in~\cite[Theorem 12]{ros1}. %%%%%%%%%%%%%%%%%%%%%%%%%%%%%%%%%%%%%%%%%%
\begin{theorem}\label{thm:drift-minorization}
Suppose a Markov chain $\{X_n\}$ with transition kernel $Q$ on a state space $\state$ 
satisfies a drift condition \eqref{eq:drift-condition}, 
and a minorization condition~\eqref{eq:minor} 
on the set $A = \V^{-1}([0, \kappa]) \subseteq \state$, 
for some $\kappa$ satisfying~\eqref{eq:cond-kappa}.  
Then the Markov chain $\{X_n\}$ has a unique stationary distribution 
$\pi$, and for any $0<r<1$ and any $n \in \bbN$, we have for any initial distribution $\mu_0$  
\begin{equation}  \label{eq:dTV-rosenthal}
d_{\mathrm{TV}}\left( \mu_0 \, Q^n ,  \pi \right) 
\leq 
(1-\epsilon)^{n r} +  
\left( \theta^{1-r} \Theta^{r}\right)^{n} \left(1+ \frac{b}{1-\beta} 
	+ \ee_{\mu_0}[\V(X_0)] \right) \ ,
\end{equation}
with $\theta$ and $\Theta$ given by~\eqref{eq:thetas}.  
\end{theorem}
%%%%%%%%%%%%%%%%%%%%%%%%%%%%%%%%%%%%%%%%%%%%%%%%%%%%%%%%%%%
%%%%%%%%%%%%%%%%%%%%%%%%%%%%%%%%%%%%%%%%%%%%%%%%%%%%%%%%%%%%
\subsection{Bounds on the convergence rates for bounded state space\label{sec:bounds-bounded}}
%%%%%%%%%%%%%%%%%%%%%%%%%%%%%%%%%%%%%%%%%%
In this section we present a proof of Theorem \ref{thm:bounded-state-rate} 
for a semistochastic process with a bounded state space.  
%%%%%%%%%%%%%%%%%%%%%%%%%%%%%%%%%%%%%%%%%%
To discretize the continuous-time process $\{X_t\}$, 
we fix a value $\Delta t > 0$ and define the Markov transition kernel $Q$ 
of the discretization $\{X_{n\,\Delta t}\}$ via 
$Q := \uu^{\Delta t}$.  
%%%%%%%%%%%%%%%%%%%%%%%%%%%%%%%%%%%%%%%%%%
To establish a uniform minorization, we first note that, for any 
nonnegative observable $f$, we can apply Lemma~\ref{lemma:u-ineq} to conclude that 
\[
[Q \, f](x) 
\geq 
\int_0^{\Delta t} \rmd s \, S(\Delta t - s, x) \, \jrt 
\left(\phi^{\Delta t-s}(x)\right) 
\int_0 ^1 \rmd \alpha \, \rho(x,\alpha) \, S(s, \alpha \phi^s( x)) \, f(\alpha \phi^s( x)) \ .
\]
Using the assumption that $0 < \lambda_* \jrt(x) \leq \lambda^{*} $, we have 
\[
S(t, x) \ \geq \ \exp(-\lambda^{*} t) \qquad \mbox{for all } x \in [0, k) \ .
\]
Combining these inequalities with the bounds on $\rho(x,\alpha)$ and $\Lambda(x)$ 
assumed in Theorem~\ref{thm:bounded-state-rate}, we arrive at
\[
[Q \, f](x) \geq \rho_{*}  \, \lambda_* \, \exp (-\lambda^{*} \Delta t) \,  \int_0^{\Delta t} \rmd s \, 
\int_0 ^1 \rmd \alpha \,  f(\alpha \, \phi^s(x)) \ .
\]
Changing the variable $\alpha$ to $z  =  \alpha \, \phi^s (x)$ 
and interchanging the order of integration, we have 
\begin{align*}
[Q \, f](x) \  & \geq  \ \rho_{*}  \, \lambda_* \exp (-\lambda^{*}  \Delta t) \,  \int_0^{\Delta t} \rmd s \, 
\int_0 ^1 \rmd \alpha  \, f(\alpha \phi^s(x)) \\
&= 
\ \rho_{*}  \lambda_*  \exp (-\lambda^{*} \Delta t) \int_0^{\Delta t} \rmd s \,  
\left( \phi^s(x) \right)^{-1}\int_ 0^{\phi^s(x)} \rmd z \,  f(z)\\
&\geq  
\ \rho_{*}  \lambda_*  \exp (-\lambda^{*}  \Delta t)   \int_0^{\Delta t} \rmd s \,  k^{-1}
\int_ 0^{\phi^ {s} (0)} \rmd z \,  f(z) \\
&= 
\ \dfrac{\rho_{*} \, \lambda_* }{k} \,  \exp (-\lambda^{*} \, \Delta t) \int_ 0^{\phi^ {\Delta t} (0)} \rmd z \,  f(z) \int_{\psi(0, z)}^{\Delta t} \rmd s \\ 
&= 
\ \dfrac{\rho_{*}  \lambda_* }{k}   \exp (-\lambda^{*} \Delta t) 
	\int_ 0^{\phi^ {\Delta t} (0)} f(z) \, [\Delta t - \psi(0, z)] \,  \rmd z \ , 
%%%%%%%%%%%%%%%%%%%%%%%%%%%%%%%%%%%%%%%%%%
\end{align*}
where have made use of the monotonicity of $\phi^t(x)$, 
the boundedness of the state space $\state = [0,k]$ (for finite~$k$), 
and the fact that $\vv(0) >0$ to arrive 
at a uniform in $x \in \state$ positive lower bound for $[Q f](x)$.  
Multiplying and dividing by $\Phi$ \eqref{eq:Phi}, 
we obtain the uniform minorization \eqref{eq:minor-1} 
with $\epsilon = \epsilon_{\Delta t}$~\eqref{eq:epsilon-delta-t} 
and minorizing measure $\eta$ \eqref{eq:minor-1} whose density is 
\[ %\begin{equation}  \label{eq:rad-nik}
\frac {\rmd \eta}{\rmd z} 
= 
\frac{\Delta t - \psi(0, z)}{\Phi} \,  \1\{0  \leq   z  \leq  \phi^{\Delta t}(0) \} \ . 
\] %\end{equation}
Figure \ref{fig:minor-domain} illustrates how the support of the minorizing 
measure is constructed and elucidates its meaning.  
Namely, for any initial value $x \in \state$, 
there is a nonzero probability that in the time interval $[0,\Delta t]$, 
a disturbance will bring the process under the the trajectory of 0 
(i.e., in the shaded region).  Once it is in the shaded region, 
the process can never leave it in the time interval $[0,\Delta t]$.  
The minorizing measure $\eta$ is due to this accumulation of probability 
in the support $[0,\phi^{\Delta t}(0)]$ of~$\eta$.  
\begin{figure}[h] 
\centerline{
\includegraphics[scale=.65]{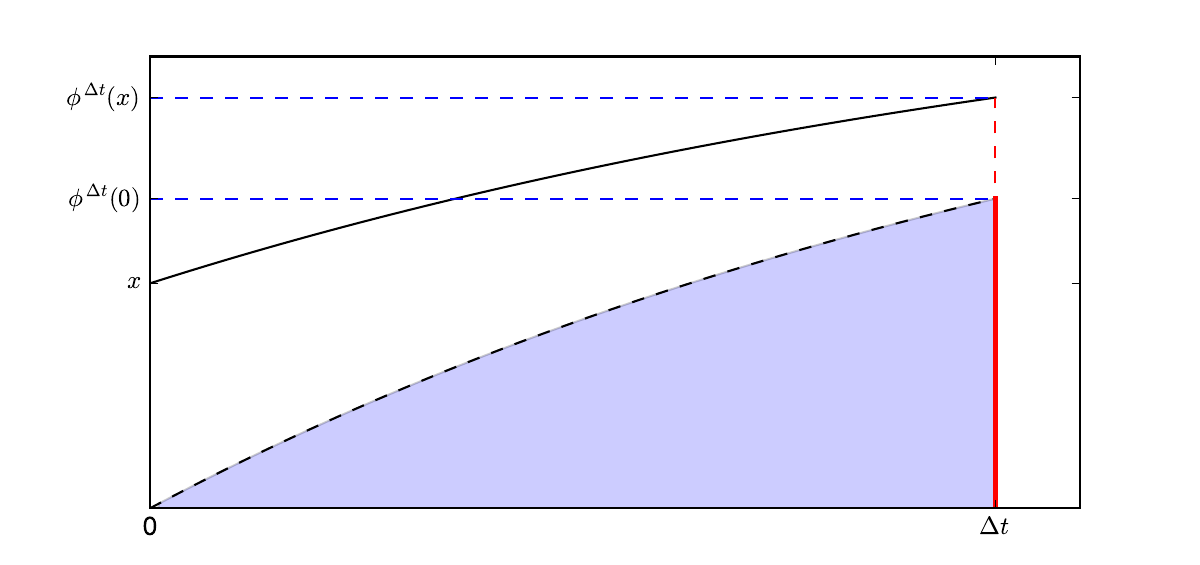}
}
\caption{On the construction of the minorizing measure in Theorem \ref{thm:bounded-state-rate}.
}
\label{fig:minor-domain}
\end{figure}
Combining the uniform minorization with Theorem~\ref{thm:uniform-minorization} 
and Lemma~\ref{lemma:continuous-to-discrete} completes the proof 
of~\ref{thm:bounded-state-rate}.  
%%%%%%%%%%%%%%%%%%%%%%%%%%%%%%%%%%%%%%%%%%%%%%%%%%%%%%%%%%%%%%%%%%%%%%%%%%%%%	
\subsection{Bounds on the convergence rates for unbounded state space\label{sec:bounds-unbounded}}
%%%%%%%%%%%%%%%%%%%%%%%%%%%%%%%%%%%%%%%%%%
As in the proof of Theorem~\ref{thm:bounded-state-rate}, 
we start the proof of Theorem~\ref{thm:unbounded-state-rate} 
by fixing a value of $\Delta t > 0$ 
and letting $Q  \ = \ \uu^{\Delta t}$ 
be the transition kernel for the corresponding discretization.  

Due to the unbounded nature of the state space in Theorem~\ref{thm:unbounded-state-rate}, 
we start by establishing a drift condition, 
i.e., an upper bound on $\ee\left[ \V( X_{\Delta t}) | X_0 = x\right]$ 
of the form~\eqref{eq:drift-condition}, for the specific choice $V(x)=I(x)$, where $I$ the identity map $I(x) = x$.  
To obtain an upper bound on $\ee[I(X_{\Delta t})|X_0]$, 
we compute $[\gen I](x)$ from~\eqref{eq:gen}: 
\begin{equation}  \label{eq:gen-1}
[\gen I](x) 
=
\vv(x)  +  \jrt(x) \int_0^1 \rho(x,\alpha)[  \alpha x   -  x ] \, \rmd \alpha 
=
\vv(x)  - \jrt(x)  \zeta(x)  x \ ,
\end{equation}
where $\zeta$ is defined by~\eqref{eq:zeta}.  
From the conditions on $\vv$ \eqref{eq:vv} and $\jrt$, 
we thus have 
\begin{equation}\label{eq:drift-ineq}
[\gen I](x) \ \leq \ \vv^* - \lambda_* \zeta(x) x \qquad 
\mbox{for all } x \in \state \ .
\end{equation}
Recall that, for any observable~$f$, the quantity 
\[
M_t := f(X_t)-f(X_0)-\int_0^t  [\gen f](X_s) \, \rmd s \ ,
\]
is a martingale.  Applying this for $f=I$, we obtain 
that, for any $t \geq  0$, 
\begin{equation}  \label{eq:EM}
\ee[M_t] 
= 
\ee \left[  \left. X_{ t} - X_0  - \int_{0}^{ t} [\gen I](X_{s}) \, \rmd s \ \right| X_0 =x\right] 
= 0 \ .
\end{equation}
Setting $u( t) = \ee[I(X_{ t})|X_0=x] =\ee[X_{ t} | X_0=x] $ 
and writing $[\gen I](X_{s})$ explicitly from \eqref{eq:gen-1}, 
we can rewrite \eqref{eq:EM} as an integral equation
\begin{equation}\label{eq:integral-drift}
u(t) 
= 
u(0) 
+ 
\int_0^t \ee \left[ \left. v(X_s)- \jrt(X_s)\zeta(X_s) X_s \right|   X_0=x \right] \, \rmd s \ .
\end{equation}
The sample paths are right-continuous, thus the right hand side of 
\eqref{eq:integral-drift} can be differentiated with respect to~$t$.  
Differentiating \eqref{eq:integral-drift} and referencing \eqref{eq:drift-ineq}, we have 
\[
u'(t) 
=
\ee \left[ \left. \vv(X_t)- \jrt(X_t)\zeta(X_t) X_t \right| X_0=x \right] 
\leq 
\vv^* - \lambda_* \zeta_* u(t) \ .
\]
Rearranging this inequality and multiplying by the integrating factor $e^{\lambda_* \zeta_* t}$ gives
\[
\frac{\rmd}{\rmd t} \left( e^{\lambda_* \zeta_* t} u(t) \right) \ \leq \ \vv^* e^{\lambda_* \zeta_* t} \ ,
\]
or, equivalently, 
\[
\frac{\rmd}{\rmd t} \left( e^{\lambda_* \zeta_* t} u(t) - \frac{\vv^*e^{\lambda_* \zeta_* t}}{\lambda_* \zeta_*}  \right)  \ \leq \ 0 \ .
\]
Therefore the expression in the parentheses is decreasing with $t$, 
so it must obtain its maximum on $[0, \infty)$ at $t=0$; recalling that $u(0)=x$, we have
\[
e^{\lambda_* \zeta_* t} u(t) - \frac{\vv^*}{\lambda_* \zeta_*} e^{\lambda_* \zeta_* t} 
\leq 
\left. 
\left( 
e^{\lambda_* \zeta_* t} u(t) - \frac{\vv^*}{\lambda_* \zeta_*} e^{\lambda_* \zeta_* t} 
\right) 
\right|_{t=0} 
= 
x - \frac{\vv^*}{\lambda_* \zeta_*} \ .
\]
Solving for $u(t)$ and setting $t = \Delta t$ produces 
the desired drift condition for the discretized process $\{X_{n\, \Delta t}\}$, 
\begin{equation}  \label{eq:drift-1}
\ee[X_{\Delta t} | X_0=x] 
\leq
e^{-\lambda_* \zeta_* \Delta t} \, x + \frac{\vv^{*}}{\lambda_* \zeta_*} 
\left( 1 \ - \ e^{-\lambda_* \zeta_* \Delta t} \right) \ ,
\end{equation}
as in~\eqref{eq:drift-condition} 
with $V=I$, $\beta = e^{-\lambda_* \zeta_* \Delta t}$ and 
$b = \frac{\vv^{*}}{\lambda_* \zeta_*} 
	\left( 1 \ - \ e^{-\lambda_* \zeta_* \Delta t} \right)$.  

Having established the drift condition, we can minorize $Q$ 
on $[0, \kappa]$ for any $\kappa < \infty$ by using the same argument 
as in the proof of Theorem~\ref{thm:bounded-state-rate}.  
In order to be able to apply Theorem~\ref{thm:drift-minorization}, 
we additionally require that $\kappa$ satisfy~\eqref{eq:cond-kappa}.  
To complete the proof of Theorem~\ref{thm:unbounded-state-rate}, 
we choose the value of $r$ in such a way that the two terms 
in the right-hand side of~\eqref{eq:dTV-rosenthal} balance each other, 
which for large~$n$ gives us 
$\left( 1-\epsilon \right)^r = \theta^{1-r} \Theta^r$, 
which gives the expression \eqref{eq:r} for~$r$.  
In particular, with this choice of $r$, 
\[
(1-\epsilon)^{n r} +  
\left( \theta^{1-r} \Theta^{r}\right)^{n} \left(1+ \frac{b}{1-\beta} 
	+ \ee_{\mu_0}[X_0] \right) 
= 
\left(2+ \frac{b}{1-\beta} + \ee_{\mu_0}[X_0] \right) (1-\epsilon)^{n r} 
\]
for all~$n$.  
Combining this with Lemma~\ref{lemma:continuous-to-discrete} 
and Theorem~\ref{thm:drift-minorization} 
completes the proof of Theorem~\ref{thm:unbounded-state-rate}.  
%%%%%%%%%%%%%%%%%%%%%%%%%%%%%%%%%%%%%%%%%%%%%%%%%%%%%%%%%%%%%%%%%%%%%%%
%%%%%%%%%%%%%%%%%%%%%%%%%%%%%%%%%%%%%%%%%%%%%%%%%%%%%%%%%%%%%%%%%%%%%%%
%%%%%%%%%%%%%%%%%%%%%%%%%%%%%%%%%%
%%%%%%%%%%%%%%%%%%%%%%%%%%%%%%%%%%
%%%%%%%%%%%%%%%%%%%%%%%%%%%%%%%%%%%%%%%%%
%%%%%%%%%%%%%%%%%%%%%%%%%%%%%%%%%%%%%%%%%
\section{\label{sec:examples}Examples}
In this section we illustrate our results on two examples.  
In both cases we assume that the jump rate $\Lambda(x)$ has a constant value~$\lambda$, 
and that the severity of disturbances is uniformly distributed, i.e., $\rho(x,\alpha) = 1$.  
We also demonstrate how one can optimize the relevant parameters 
$\Delta t$ and $\kappa$ in order to obtain tighter bound on rates of convergence. 
%%%%%%%%%%%%%%%%%%%%%%%%%%%%%%%%%%%%%%%%%%
\subsection{Example: bounded state space\label{sec:example1}}
%%%%%%%%%%%%%%%%%%%%%%%%%%%%%%%%%%%%%%%%%%
In this example we consider a model of growth with saturation 
on $\state = [0, k]$: 
\[
x'(t) = k - x \ , \qquad k=\mathrm{const}>0 \ .
\]
In this case (cf.~\eqref{eq:phi-psi}), 
\[
\phi^t(x) = k  +  (x - k)e^{-t} \ , \qquad 
\psi(x_0 , x) =  \ln \frac{k  -  x_0}{k  -  x} \ .
\] 
From Theorem \ref{thm:bounded-state-rate}, for fixed  $\Delta t$ 
and arbitrary initial distribution $\mu_0$, the following bound holds
\[
d_{\mathrm{TV}}\left( \mu_0 \, \uu^t , \pi \right) 
\leq 
\left(1-\epsilon_{\Delta t}\right)^{\lfloor t /\Delta t \rfloor} 
\]
($\pi$ is the unique stationary distribution).  
We have 
\[
\Phi 
= 
\int_ 0^{k (1-e^{-\Delta t})} \left(\Delta t - \ln  \frac{k}{k \, - \, z} \right) \, \rmd z
=
k( \Delta t \, + \, e^{-\Delta t} - 1   ) \ , 
\]
\[
\epsilon_{\Delta t} 
= 
\frac{ \Phi \, \lambda\,  e^{-\lambda\, \Delta t} } { k }   
= 
\lambda\, e^{-\lambda\, \Delta t}  ( \Delta t \, + \, e^{-\Delta t} - 1   ) \ .
\]
For convergence rates, the quantity of interest is 
$(1 - \epsilon_{\Delta t})^{1/\Delta t}$ (cf.~\eqref{eq:dTV-bounded}).  
For concreteness, take $\lambda=1$.  
In Figure~\ref{fig:rate1}, we plot 
$(1 - \epsilon_{\Delta t})^{1/\Delta t}$ as a function of $\Delta t$ 
and observe that it exhibits a minimum at $\Delta t \approx 0.82$, 
for which $\epsilon_{\Delta t} \approx 0.115$.  
The intuitive reason for existence of such an optimal value of $\Delta t$ 
was discussed in Remark~\ref{rem:choices-bounded}.  
\begin{figure}[h] 
\centerline{
\includegraphics[scale=.65]{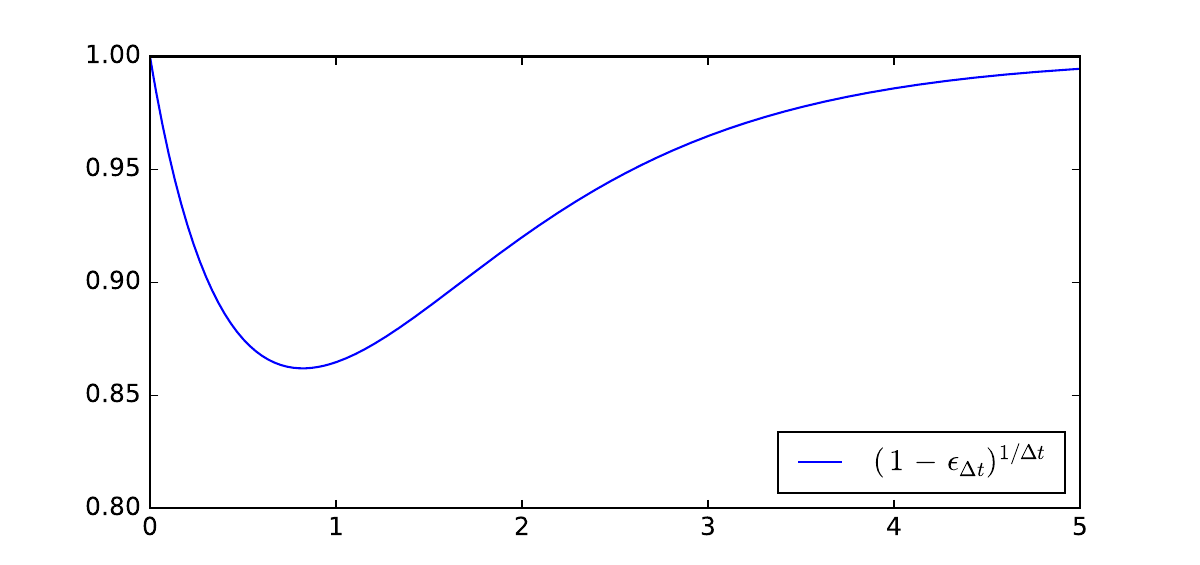}
}
\caption{Plot of $(1 - \epsilon_{\Delta t})^{1/\Delta t}$ vs.\ $\Delta t$.
}
\label{fig:rate1}
\end{figure}
Choosing $\Delta t =0.82$, 
we obtain that, for any initial distribution $\mu_0$, 
\[
d_{\mathrm{TV}}( \mu_t , \pi ) \leq ( 1 - 0.115 )^{\lfloor t/0.82 \rfloor} 
\leq 
1.13 \, e^{ - 0.148 \, t} \ .
\]
For comparison, in Figure~\ref{fig:rate2} 
we plot $(1 - \epsilon_{\Delta t})^{\lfloor t/\Delta t \rfloor}$ 
as a function of $t$ for several values of $\Delta t$.  
\begin{figure}[h] 
\centerline{
\includegraphics[scale=.65]{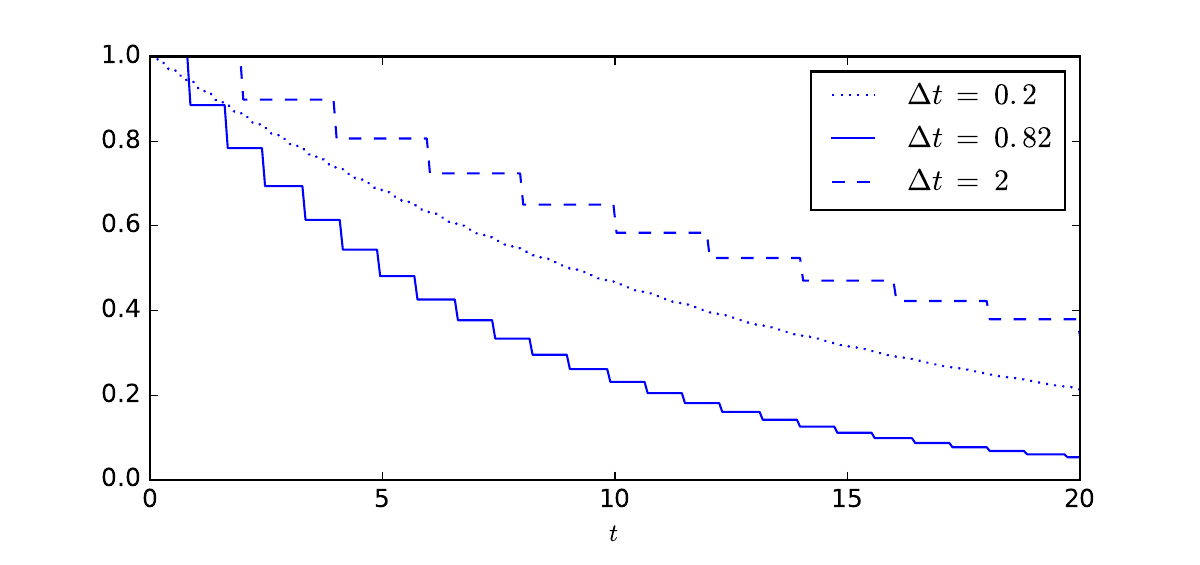}
}
\caption{Plots of $(1 - \epsilon_{\Delta t})^{\lfloor t/\Delta t \rfloor}$ vs.\ $t$ 
for selected values of $\Delta t$.
}
\label{fig:rate2}
\end{figure}
%%%%%%%%%%%%%%%%%%%%%%%%%%%%%%%%%%%%%%%%%%%%
%%%%%%%%%%%%%%%%%%%%%%%%%%%%%%%%%%%%%%%%%%%%
\subsection{Example: unbounded state space\label{sec:example-unbounded}}
%%%%%%%%%%%%%%%%%%%%%%%%%%%%%%%%%%%%%%%%%%
Consider the case of constant growth rate on $\state = [0,\infty)$: 
\[
x'(t) = v = \mathrm{const} > 0 \ .
\]
Our flow and time-duration functions are
\[
\phi^t(x) = x  +  t \ , \qquad 
\psi(x_0 , x)  = x - x_0 \ .
\] 
From Theorem \ref{thm:unbounded-state-rate}, for fixed  $\Delta t> 0$, 
we can first establish a drift condition using the identity as our drift function.  
In this case the average fractional loss $\zeta(x)= \frac{1}{2}$ 
does not depend on $x$, so we can compute the expectation exactly,
\[
\ee[X_{\Delta t} | X_0 = x] = e^{-\lambda \, \Delta t / 2} 
+ \frac{2 v} {\lambda}\left( 1 -  e^{-\lambda \, \Delta t / 2} \right) \ , 
\]
so the drift parameters are $\beta = e^{-\lambda \, \Delta t / 2}$, 
$b = \frac{2 v} {\lambda}\left(  1  -  e^{-\lambda \, \Delta t / 2}  \right)$.  
To compute bounds on the convergence rates, we need to select a value 
$\kappa \ > \  \frac{2b}{1 - \beta} \ = \frac{4v}{\lambda}$ 
for which we minorize the process on $[0, \kappa]$.  
We easily obtain 
$\Phi = v ( \Delta t ) ^2$ and 
$\epsilon_{\Delta t, \kappa} = \frac{v (\Delta t)^2 \, \lambda \, e^{-\lambda \Delta t }}{\kappa}$, 
where we are emphasizing the dependence of $\epsilon$ on $\kappa$ as well as $\Delta t$. 
For $\theta$ and $\Theta$ \eqref{eq:thetas} we obtain 
\[
\theta = \frac{  1 + \frac{4v}{\lambda} + \left(\kappa - \frac{4v}{\lambda}\right) e^{-\lambda \Delta t /2}  }{1+\kappa} \ ,
\qquad
 \Theta = 1 + \frac{4v}{\lambda} + \left(2\kappa - \frac{4v}{\lambda}\right)  e^{-\lambda \Delta t /2} \ ;
\]
in the expression for $\theta$, note that the restriction on $\kappa$ ensures the positivity 
of the exponential term in the numerator.
Continuing the example with   
the specific values $v = 1$ and $\lambda = 2$, 
we obtain $\beta \approx 0.405$ and $b \approx 0.595$.  
We can then make appropriate choices for $\Delta t $ and $\kappa$ by minimizing the expression
\[
\left( 1-\epsilon_{\Delta t, \kappa} \right)^{\frac{r(\Delta t, \kappa)}{\Delta t}}
\]
as illustrated in Figures~\ref{fig:rate5} and~\ref{fig:rate6}.  
The dependance of this expression on $\Delta t$ and $\kappa$ is in accordance with 
our reasoning in Remarks \ref{rem:choices-bounded} and~\ref{rem:choices}.  
%%%%%%%%%%%%%%%%%%%%%%%%%%%%%%%%%%%%%%%%%%
\begin{figure}[h] 
\centerline{
\includegraphics[scale=.65]{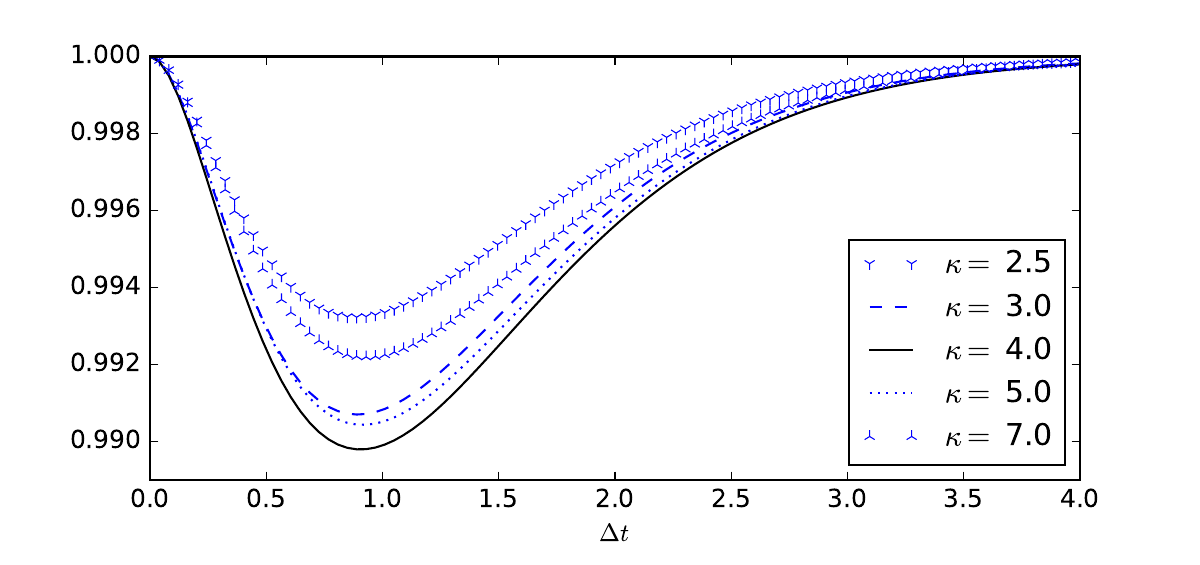}
}
\caption{Plots of $(1 - \epsilon_{\Delta t,\kappa})^{r / \Delta t }$ vs.\ $\Delta t$ for selected $\kappa$.
}
\label{fig:rate5}
\end{figure}

\begin{figure}[h] 
\centerline{
\includegraphics[scale=.65]{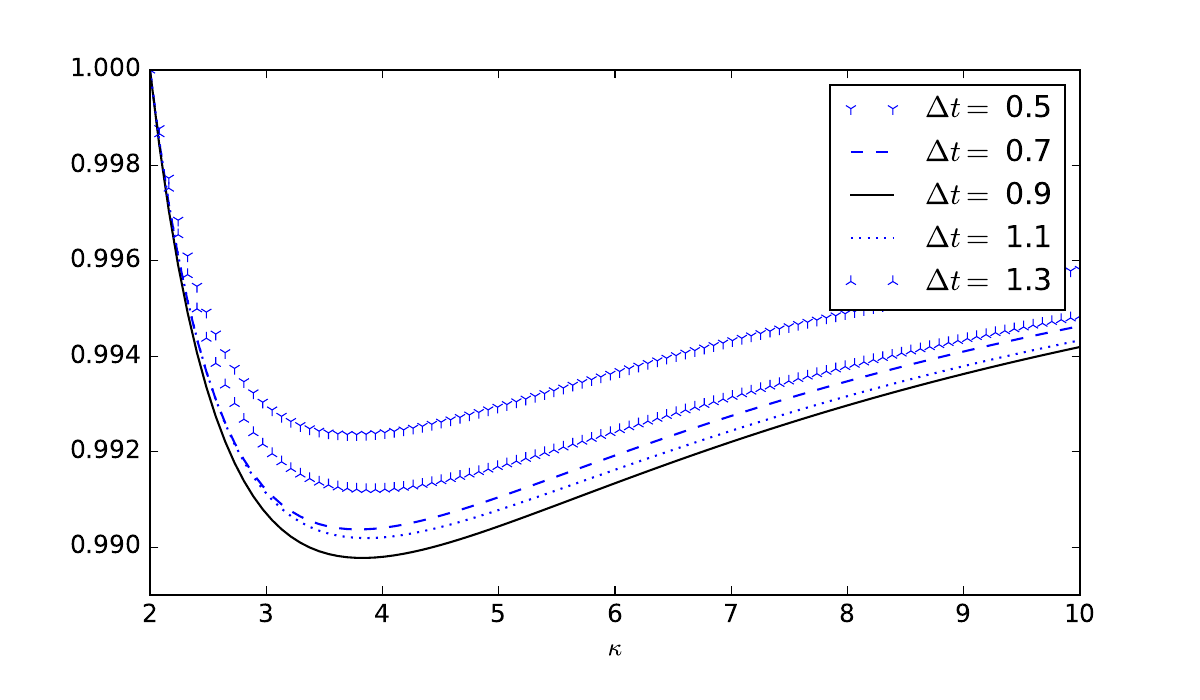}
}
\caption{Plots of $(1 - \epsilon_{\Delta t,\kappa})^{r / \Delta t }$ vs.\ $\kappa$ for selected $\Delta t$.  
}
\label{fig:rate6}
\end{figure} 
Consequently, we choose $\Delta t = 0.904$, $\kappa = 3.83$, and $r$ as in \eqref{eq:r} to obtain
\begin{align*}
d_{\mathrm{TV}}\left( \mu_t , \pi \right) \ &\leq \ C (1 - 0.070)^{ r\lfloor t / 0.904\rfloor} \\
& \leq
\ 1.02 \, C \, e^{-0.014 \, t}
\ ,
\end{align*}
with $C = 3 + \ee_{\mu_0} [X_0]$. 
%%%%%%%%%%%%%%%%%%%%%%%%%%%%%%%%%%%%%%%%%
%%%%%%%%%%%%%%%%%%%%%%%%%%%%%%%%%%%%%%%%%
\bibstyle{plain}

\end{document}